\documentclass{amsart}

\usepackage{amsmath, amssymb, amsthm}
\usepackage[colorlinks=true]{hyperref}
\usepackage{cite}
\usepackage{arydshln}

\DeclareMathOperator{\BS}{BS}

\DeclareMathOperator{\laspan}{span}
\DeclareMathOperator{\Aut}{Aut}

\DeclareMathOperator{\im}{im}
\DeclareMathOperator{\ind}{ind}

\DeclareMathOperator{\GL}{\bf GL}

\newtheorem{theorem}{Theorem}[section]

\newtheorem{proposition}[theorem]{Proposition}
\newtheorem{corollary}[theorem]{Corollary}
\newtheorem{definition}[theorem]{Definition}
\newtheorem{fact}[theorem]{Fact}

\title{Irreducible Representations of Baumslag-Solitar~Groups}

\author{Daniel McLaury}
\address{Department of Mathematics, The University of Oklahoma, Norman, OK, USA}
\email{daniel.mclaury@gmail.com}

\date{}

\begin{document}

\maketitle

\begin{abstract}
We classify the finite-dimensional irreducible linear representations of the \emph{Baumslag-Solitar groups} $\BS(p, q) = \langle a, b \; | \; ab^p = b^q a \rangle$ for relatively prime $p$ and $q$.  The general strategy of the argument is to consider the matrix group given by image of a representation and study its Zariski closure in $\GL_n$.
\end{abstract}

\section{Baumslag-Solitar Groups}

A group $G$ is said to be \emph{residually finite} if, for each $g \in G$, there is a finite group $H$ and a homomorphism $h : G \rightarrow H$ such that $h(g) \neq 1$.  It is said to be \emph{Hopfian} if each epimorphism $\phi : G \twoheadrightarrow G$ is an automorphism; equivalently, a group is non-Hopfian if it is isomorphic to one of its proper quotients.  Residually finite groups are Hopfian, but not conversely.  For nonzero integers $p$ and $q$, the family of \emph{Baumslag-Solitar groups} 
\begin{equation}
\BS(p, q) = \langle a, b \; | \; ab^p = b^q a \rangle.
\end{equation}
was defined in\cite{BaumslagSolitar} to provide examples of finitely presented non-Hopfian groups.  Note that $BS(p, q) = BS(-p, -q)$, and further that $\BS(p, q) \cong \BS(q, p)$ by $a \mapsto a^{-1}$.  The following definition helps classify the Baumslag-Solitar subgroups:
\begin{definition}
Positive integers $p$ and $q$ are said to be \emph{meshed} if either
\begin{enumerate}
\item $p|q$ or $q|p$, or
\item $p$ and $q$ have precisely the same prime divisors.
\end{enumerate}
\end{definition}
\begin{theorem}
(\cite{BaumslagSolitar}, Theorem 1; \cite{MeskinNonresiduallyFinite}, Theorem C)
Let $p$ and $q$ be nonzero integers.  Then $\BS(p, q)$ is 
\begin{enumerate}
  \item residually finite, and thus Hopfian, if $|p| = |q|$ or $|p| = 1$ or $|q| = 1$;
  \item Hopfian if $p$ and $q$ are meshed;
  \item non-Hopfian if $p$ and $q$ are not meshed.
\end{enumerate}
\end{theorem}
As finitely-generated matrix groups are known to be Hopfian (see \cite{MalcevFaithfulRepInfGrp}), the groups $\BS(p, q)$ with $p$ and $q$ unmeshed do not have any faithful representations.

The following fact about Baumslag-Solitar groups will be useful:
\begin{proposition}
\label{propbpkbqk}
For any integer $k$, $b^{p^k} = a^{-k} b^{q^k} a^k$ in $\BS(p, q)$.
\end{proposition}
\begin{proof}
Rewrite the Baumslag-Solitar relation as $b^p = a^{-1} b^q a$ and take the $j$-th power of each side, giving
\begin{equation}
b^{pj} = (a^{-1} b^q a)^j = a^{-1} b^{qj} a
\end{equation}
Now apply this to $b^{p^k}$ repeatedly.  In the case $k > 0$,
\begin{align}
b^{p^k} &= a^{-1} b^{q p^{k-1}} a \\
	&= a^{-2} b^{q^2 p^{k-2}} a^2 \\
	&\phantom{=}\vdots \notag \\
	&= a^{-k} b^{q^k} a^k;
\end{align}
the $k < 0$ case is analogous, and the $k = 0$ case is immediate.
\end{proof}
As a reference for representations of non-Hopfian Baumslag-Solitar groups, see \cite{GoodmanPhD}.  There, Goodman shows the existence of some particular representations of $\BS(p, q)$ when $p$ and $q$ are relatively prime, and characterizes the geometry of the variety of $n$-dimensional representations of $\BS(p, q)$ at these points.

\section{Notation}

Fix nonzero, relatively prime integers $p$ and $q$, not both $\pm 1$, and let 
\begin{equation}
\Gamma = \BS(p, q) = \langle a, b \; | \; a b^p = b^q a \rangle.
\end{equation}
Fix also an $n$-dimensional complex $\Gamma$-module $V$, and write $\rho : \Gamma \rightarrow \GL_n(\mathbb{C})$ for the corresponding representation.  Let $G$ denote the image of $\rho$, and write $A$ and $B$ for $\rho(a)$ and $\rho(b)$, respectively.  Finally, let $H$ denote the subgroup of $G$ generated by $B$.

\section{The Structure of $G$}

Write ${\bf G} = \overline{G}$ and ${\bf H} = \overline{H}$ for the Zariski closures of $G$ and $H$ as subsets of $\GL_n(\mathbb{C})$.  We assume the standard results about linear algebraic groups, as found in chapters 1--3 of \cite{SpringerLinAlgGrp}.  In particular, recall that:
\begin{fact}
\label{lemmaFI}
Let ${\bf G}$ be an algebraic group. If $K \leq H \leq \bf G$ with $[H:K]$ finite,
\begin{enumerate}
  \item $\left[\overline{H} : \overline{K}\right]$ is finite as well,
  \item $\left[\overline{H} : \overline{K}\right]$ divides $[H:K]$,
  \item $\overline{K}^0 = \overline{H}^0$.
\end{enumerate}
\end{fact}

Key to the classification is that, while $\langle b \rangle$ is not a normal subgroup of $\Gamma$, its image $H$ will be normal in $G$.  This fact is a consequence of the following two theorems:

\begin{theorem}
\label{thmhnormg}
${\bf H} \trianglelefteq {\bf G}$.
\end{theorem}
\begin{proof}
It suffices to check that $\bf H$ is normalized by $A$, as $B \in \bf H$ and $A$ and $B$ together generate $\bf G$ as an algebraic group.  By the Baumslag-Solitar relation, 
\begin{equation}
A B^p A^{-1} = B^q.
\end{equation}
In turn, we see that
\begin{align}
A \langle B^p \rangle A^{-1} &= \langle B^q \rangle, \\
A \overline{\langle B^p \rangle} A^{-1} &= \overline{\langle B^q \rangle},\\
A \overline{\langle B^p \rangle}^0 A^{-1} &= \overline{\langle B^q \rangle}^0.
\end{align}
Applying Fact~\ref{lemmaFI}(iii) to $\langle B^k \rangle \leq \langle B \rangle \leq {\bf G}$, we have $\overline{\langle B^k \rangle}^0={\bf H}^0$ for all $k$, so $A {\bf H}^0 A^{-1} = {\bf H}^0$.  Conjugation by $A$ therefore induces an automorphism of the finite group ${\bf H}/{\bf H}^0$ which takes $\overline{\langle B^p \rangle}  / {\bf H}^0$ to $\overline{\langle B^q \rangle}  / {\bf H}^0$; in particular, these two subgroups have the same order, and consequently the same index.   By Fact~\ref{lemmaFI}(ii), the former has index dividing $p$, whereas the latter has index dividing $q$.  As $p$ and $q$ are relatively prime, both $\overline{\langle B^p \rangle}  / {\bf H}^0$ and $\overline{\langle B^q \rangle}  / {\bf H}^0$ must have index one in ${\bf H}/{\bf H}^0$, so $\overline{\langle B^q \rangle} = \overline{\langle B^p \rangle} = {\bf H}$.  Now 
\begin{equation}
 A {\bf H} A^{-1} = A \overline{\langle B^p \rangle} A^{-1} = \overline{\langle B^q \rangle} = \bf H,
\end{equation}
so $A$ normalizes $\bf H$ as desired.
\end{proof}

\begin{samepage}
\begin{theorem}
\label{thmHFinite}
If $\rho$ is irreducible, then
\begin{enumerate}
  \item $\bf H$ is diagonalizable,
  \item $\bf H$ is a finite cyclic group generated by $B$, and
  \item the order $\ell$ of $B$ divides $p^{\varphi(\ell)} - q^{\varphi(\ell)}$; in particular, $(\ell, p) = (\ell, q) = 1$.
\end{enumerate}
\end{theorem}
\end{samepage}

\begin{proof}
(i) Since $\langle B \rangle$ is commutative, so is its closure $\bf H$.  Recall that a commutative linear algebraic group is the direct sum of its unipotent and semisimple parts, each of which is moreover a characteristic subgroup.  Therefore ${\bf H}_u \trianglelefteq \bf G$, so we may consider the $\bf G$-submodule $V^{{\bf H}_u}$ of points fixed by ${\bf H}_u$.  Since ${\bf H}_u$ is unipotent, $V^{{\bf H}_u}$ is nonzero.  But $V$ was simple by assumption, so $V^{{\bf H}_u} = V$, which can only be the case if ${\bf H}_u = 1$.  We must then have ${\bf H} = {\bf H}_s$, so $\bf H$ is a commutative group of diagonalizable matrices, and therefore diagonalizable. (ii) By the rigidity of diagonalizable subgroups, $N_{\bf G}({\bf H}) / Z_{\bf G}({\bf H})$ is finite, so conjugation by $A$ gives a finite-order automorphism of $\bf H$.   Let $r$ denote this order. By Proposition~\ref{propbpkbqk} we have $B^{p^r} = A^{-r} B^{q^r} A^r$, so $B^{p^r} = B^{q^r}$, i.e.\ $B^{p^r - q^r} = 1$.  Therefore $\langle B \rangle$ is finite.  Consequently, $\langle B \rangle$ is already closed, so\ ${\bf H} = \langle B \rangle$. (iii) If $\ell = 1$ this is immediate, so suppose $\ell > 1$.  As $\bf H$ is cyclic of order $\ell$, $\lvert\Aut {\bf H}\rvert = \varphi(\ell)$, so $r | \varphi(\ell)$.  By the argument from (ii),  $B^{p^{\varphi(\ell)} - q^{\varphi(\ell)}} = 1$, so $\left.\ell \;\middle|\;  p^{\varphi(\ell)} - q^{\varphi(\ell)}\right.$.  The rest is elementary number theory.
\end{proof}

To summarize, we've now shown that every irreducible representation of $\Gamma$ factors through a metacyclic group, which is moreover cylic-by-\emph{finite}-cyclic.  The following calculation will be useful in the next section:

\begin{proposition}
\label{propBtothes}
In $G$, we have the identities
\begin{enumerate}
\item  $A^{-1} B A = B^s$, and 
\item $A^{-i} B^j A^i = B^{js^i}$ in general,
\end{enumerate}
where $\ell = |B|$ and $p \equiv qs \pmod{\ell}$.
\end{proposition}
\begin{proof}
(i) Since $H \trianglelefteq G$, $A^{-1}BA \in H$; as $H = \langle B \rangle$, this means $A^{-1} BA = B^s$ for some s.  
Now
\begin{equation}
B^p = A^{-1} B^q A = (A B A^{-1})^q = (B^s)^q = B^{qs};
\end{equation}
as $\ell = |B|$ by definition, we have $p \equiv qs \pmod{\ell}$.  (ii) is immediate from (i).
\end{proof}

\section{A Basis for $V$}

From now on, assume $\rho$ is irreducible, so that $V$ is a simple $\Gamma$-module and the results of Theorem~\ref{thmHFinite} will apply.  (We will not need to make use of algebraic groups any further.)   As in the previous section, let $\ell$ denote the order of $B$ and $s$ be the unique solution to $p \equiv q s \pmod{\ell}$.  The next step is to get a canonical basis for $V$ of eigenvectors of $B$ which behaves nicely under multiplication by $A$.

\begin{proposition}
\label{propEigenB}
Let $v$ be a nonzero $\lambda$-eigenvalue of $B$ for some $\lambda$.
\begin{enumerate}
  \item $A^mv$ is a $\lambda^{s^m}$-eigenvector of $B$.
  \item $\{v, Av, A^2v, \ldots, A^nv\}$ is a basis of $B$-eigenvectors.
  \item The eigenvalues of $B$ are distinct primitive $\ell$-th roots of unity.
\end{enumerate}
\end{proposition}
\begin{proof}
(i) By Proposition~\ref{propBtothes}(ii), $B(A^m v) = A^m (B^{s^m} v) = \lambda^{s^m} (A^m v)$. (ii) Let $N$ be the smallest positive integer such that $A^{N+1}v \in \laspan \{v, Av, \ldots A^N v\}$.  Then $W = \laspan \{v, Av, \ldots A^N v\}$ is a nonempty subspace of $V$ invariant under $a$ and $b$, and consequently under all of $\Gamma$.  Since $V$ is a simple $\Gamma$-module, we must have $V = W$.  This set of $N+1$ vectors is linearly independent by construction, so it's a basis for $V$.  Equating dimensions, $N + 1 = \dim V = n + 1$.  (iii) We've now shown each eigenvalue of $B$ is a power of $\lambda$.  If $\lambda$ had order $\ell' < \ell$, then each power of $\lambda$ would have order dividing $\ell'$.  As $B$ is diagonalizable, this would mean $B^{\ell'} = I$.  So $\lambda$ is a primitive $\ell$-th root of unity.  But $\lambda$ was an arbitrary eigenvalue of $B$, so each eigenvalue of $B$ is a primitive $\ell$-root of unity. Finally, if $\lambda^{s^m} = \lambda^{s^{m'}}$ for some $0 \leq m < m' \leq n$, then 
$\lambda^{s^{m' - m}} = 1$, contradicting the previous sentence.
\end{proof}

\begin{samepage}
\begin{corollary}
$A^{n+1}$ stabilizes each eigenspace of $B$.
\end{corollary}
\begin{proof}
Let $v$ be a $\lambda$-eigenvector of $B$.   By the proposition $A^{n+1} v$ is a $\mu$-eigenvector of $B$ for some $\mu$.  Consider the action of $A$ on $V$.  Each of $Av, A(Av), \ldots A(A^{n-1})v$ is an eigenvector of $B$ corresponding to some non-$\lambda$ eigenvalue.  If $\mu \neq \lambda$, then we would have $\im A \subseteq \laspan \{Av, A^2v, \ldots A^n v\}$, which is a proper subspace of $V$.  Since $A$ is invertible, this can't be the case, so $A^{n+1} v$ must be a $\lambda$-eigenvector of $B$.
\end{proof}
\end{samepage}

Interpreting these as statements about the matrices of $A$ and $B$ in a particular basis, we've essentially proven the following result:

\begin{proposition}
\label{propCanonicalFormAB}
There is a basis for $V$ in which
\begin{equation}
A = c \left[ \begin{array}{cccc:c}
0 &            &             &           &  \,1 \\
1 & \ddots &            &            &  \,0 \\
   & \ddots & \ddots &            &  \vdots \\
  &              & \ddots & 0 &  \vdots \\
  &             &             &   1        & \,0
\end{array} \right], \quad B = \left[ \begin{array}{ccccc}
\lambda \\
& \lambda^s \\
 & & \ddots \\
 & & & \ddots \\ 
 & & & & \lambda^{s^n} \end{array}\right]
\label{eqCanonicalFormAB}
\end{equation}
for some nonzero complex number $c$ and some root $\lambda$ of unity.
\end{proposition}
\begin{proof}
Rescale the basis from Proposition~\ref{propEigenB}(ii).
\end{proof}

We can cast these results in more representation-theoretic terms.  Suppose we have a representation of the form (\ref{eqCanonicalFormAB}) with $c = 1$.  Let $\eta = \langle a^{n+1}, b \rangle \leq \Gamma$.  Then we have $\rho=\ind^\Gamma_\eta \chi$, where $\chi : \eta \rightarrow \mathbb{C}^\times$ is the character of $\eta$ sending $b$ to $\lambda$ and $a$ to 1.  In principle, we could now complete the classification of simple $\Gamma$-modules by applying the Mackey irreducibility criterion to see which of these characters induce irreducible representations of $\Gamma$.  As it turns out, though, the argument from first principles in the next section is simpler.

\section{Classification of simple $\Gamma$-modules}

In Proposition~\ref{propCanonicalFormAB}, we showed that each $n+1$-dimensional simple representation of $\Gamma$  is conjugate to one of the form (\ref{eqCanonicalFormAB}).  In this section, we completely characterize such representations, thereby giving a complete description of the irreducible representations of $\Gamma$.

\begin{theorem}
\label{thmRepExistConditions}
Let $c$ be a nonzero complex number, $\ell$ a positive integer, $\lambda$ a primitive $\ell$-th root of unity, and $p \equiv q s\pmod{\ell}$.  There is an $(n+1)$-dimensional representation of $\Gamma$ sending $a$ to $A$ and $b$ to $B$, where $A$ and $B$ are as in (\ref{eqCanonicalFormAB}), if and only if $\ell$ divides $q^{n+1} - p^{n+1}$.
\end{theorem}
\begin{proof}
Note that this condition implies that $\ell$ is relatively prime to both $p$ and $q$.  We just need to check when $A B^p = B^q A$.  Let $\{e_0, e_1, \ldots, e_n\}$ denote the standard basis of $\mathbb{C}^{n+1}$.   For $0 \leq i \leq n$,
\begin{equation}
AB^p e_i = A \left( \lambda^{ps^i} e_i \right) = \begin{cases}
c \lambda^{ps^i} e_{i+1} & \text{ if } i < n, \\
c \lambda^{ps^n} e_0 & \text{ if } i = n;
\end{cases}
\end{equation}
\begin{align}
B^q A e_i &= \begin{cases}
B^q c e_{i+1} & \text{ if } i < n, \\
B^q c e_0 & \text{ if } i = n,
\end{cases} \\
&= \begin{cases}
c \lambda^{q s^{i+1} } e_{i+1} & \text{ if } i < n, \\
c \lambda^q e_0 & \text{ if } i = n.
\end{cases}
\end{align}
For $i \leq n$, the condition $c \lambda^{ps^i} = c \lambda^{q s^{i+1} }$ is always satisfied, as it is equivalent to $p \equiv qs \pmod{\ell}$.  For $i = n$, the condition $c \lambda^{p s^n} = c \lambda^q$ is equivalent to $p s^n \equiv q \pmod{\ell}$.  Multiplying through by $q^n$, we have $p q^n s^n \equiv q^{n+1} \pmod{\ell}$.  As $p \equiv qs \pmod{\ell}$, this just says that $p^{n+1} \equiv q^{n+1} \pmod{\ell}$, or in other words that $\ell$ divides $q^{n+1} - p^{n+1}$.
\end{proof}
As an example, consider three-dimensional representations of $\BS(2, 5)$ of this form.  We want to pick some value for $\ell$ which divides
\begin{equation}
5^3 - 2^3 = 125 - 8 = 117 = 3^2\cdot 13.
\end{equation}
When $\ell = 3$, the solution to $2 \equiv 5s \pmod{3}$ is $s = 1$.  The primitive third roots of unity are $-\frac{1}{2} \pm \frac{\sqrt{3}}{2}$, so we get the representations
\begin{equation}
a \mapsto \begin{pmatrix}
0 & 0 & c \\
c & 0 & 0 \\
0 & c & 0
\end{pmatrix}, \qquad b \mapsto \begin{pmatrix}
-\frac{1}{2} \pm \frac{\sqrt{3}}{2} & 0 & 0 \\
0 & -\frac{1}{2} \pm \frac{\sqrt{3}}{2} & 0 \\
0 & 0 & -\frac{1}{2} \pm \frac{\sqrt{3}}{2}
\end{pmatrix}.
\end{equation}
When $\ell = 9$, the solution to $2 \equiv 5s \pmod{9}$ is $s = 4$.  If $\zeta$ is a primitive ninth root of unity, then we get the representations
\begin{equation}
a \mapsto \begin{pmatrix}
0 & 0 & c \\
c & 0 & 0 \\
0 & c & 0
\end{pmatrix}, \qquad b \mapsto \begin{pmatrix}
\zeta & 0 & 0 \\
0 & \zeta^4 & 0 \\
0 & 0 & \zeta^7
\end{pmatrix}.
\end{equation}
Notice that the former are reducible while the latter are not.  All that remains is to distinguish these two cases.  From Proposition~\ref{propEigenB}(iii), we know that it's a necessary condition that the eigenvalues $\lambda, \lambda^s, \ldots, \lambda^{s^n}$ be distinct.  Clearly this condition is sufficient as well -- if it holds, then $b$ fixes each $\laspan \{e_i\}$, while $a$ permutes them in a single orbit.  

\begin{corollary}
\label{corRepSimpleConditions}
Such a representation is irreducible if and only if $\ell$ does not divide $q^k - p^k$ for any $1 \leq k \leq n$.
\end{corollary}
\begin{proof}
We've seen already that irreducibility is equivalent to $B$ having distinct eigenvalues.  If $B$ does not have distinct eigenvalues, then $\lambda^{s^i} = \lambda^{s^j}$ for some $i$ and $j$ with $i \not \equiv j \pmod{\ell}$, or equivalently $s^i \equiv s^j \pmod{\ell}$.  This may be rewritten as $s^{i-j} \equiv 1 \pmod{\ell}$; multiplying through by $q^{i-j}$, we have $p^{i - j} \equiv q^{i - j} \pmod{\ell}$, i.e.\ $\ell$ divides $q^{i-j} - p^{i-j}$.  Finally, notice that the difference $i - j$ can take on any value between 1 and $n$.
\end{proof}

Applying this to the example above, we see that
\begin{equation}
5^2 - 2^2 = 25 - 4 = 21 = 3 \cdot 7, \qquad 5^1 - 2^1 = 3,
\end{equation}
which is in accord with our observation that taking $\ell = 9$ gave an irreducible representation, while taking $\ell = 3$ did not.

\section{Conclusion}

It is hoped that these results can be generalized to the case of arbitrary Baumslag-Solitar groups, which contain the groups examined here as subgroups: if $\Gamma = \BS(m,n)$, where $(m, n) = d$, then $\Delta = \langle a, b^d \rangle \leq \Gamma$ is isomorphic to $\BS(p, q)$, where $p = m/d$ and $q = n/d$.  Note that $\Delta$ is \emph{not}, in general, of finite index in $\Gamma$.

This paper was prepared under the direction of Andy Magid, whose supervision was invaluable in finding and solving the problem and submitting it for publication.  The author is also grateful to Nikolay Buskin for helpful discussions.

\bibliography{references}{}
\bibliographystyle{alpha}

\end{document}